\documentclass[a4paper,12pt]{amsart}
\usepackage[dvips]{graphicx}
\usepackage{amssymb}
\usepackage{amscd}

\newtheorem{Theorem}{Theorem}[section]
\newtheorem{Lemma}[Theorem]{Lemma}
\newtheorem{Corollary}[Theorem]{Corollary}
\newtheorem{Proposition}[Theorem]{Proposition}
\theoremstyle{definition}
\newtheorem{Definition}[Theorem]{Definition}
\newtheorem{Example}[Theorem]{Example}
\theoremstyle{remark}
\newtheorem{Remark}{Remark}

\font\sy=cmsy10

\font\ym=msbm10  
\newcommand{\cA}{{\hbox{\sy A}}}

\newcommand{\cC}{{\hbox{\sy C}}}
\newcommand{\cD}{{\hbox{\sy D}}}

\newcommand{\cH}{{\hbox{\sy H}}}

\newcommand{\cK}{{\hbox{\sy K}}}
\newcommand{\cL}{{\hbox{\sy L}}}
\newcommand{\cM}{{\hbox{\sy M}}}
\newcommand{\cN}{{\hbox{\sy N}}}
\newcommand{\cO}{{\hbox{\sy O}}}
\newcommand{\cP}{{\hbox{\sy P}}}
\newcommand{\cQ}{{\hbox{\sy Q}}}
\newcommand{\cR}{{\hbox{\sy R}}}
\newcommand{\cS}{{\hbox{\sy S}}}
\newcommand{\cT}{{\hbox{\sy T}}}

\newcommand{\cV}{{\hbox{\sy V}}}
\newcommand{\cW}{{\hbox{\sy W}}}

\newcommand{\C}{{\text{\ym C}}}

\newcommand{\N}{\text{\ym N}}

\newcommand{\R}{\text{\ym R}}

\newcommand{\End}{\hbox{\rm End}}
\newcommand{\Hom}{\hbox{\rm Hom}}

\title[Multicategories of planar diagrams]
{Representations of multicategories of planar diagrams and tensor categories}
\author[Shigeru Yamagami]{}

\begin{document}
\maketitle
\begin{center}
Shigeru YAMAGAMI

\end{center}
\begin{center}
Graduate School of Mathematics
\end{center}
\begin{center}
Nagoya University 
\end{center}
\begin{center} 
Nagoya, 464-8602, JAPAN 
\end{center}    

                            
\section*{}
We shall discuss how the notions of multicategories and their linear representations 
are related with tensor categories. 
When one focuses on the ones arizing from planar diagrams, 
it particularly implies that there is a natural one-to-one correspondence between 
planar algebras and singly generated bicategories. 

\section{Multicategories}


\textbf{Multicategory} is a categorical notion which concerns a class of objects and 
morphisms so that morphisms are
enhanced to admit multiple objects as inputs, whereas outputs are kept to be single. 
The operation of composition can therefore be performed in a ramified way, 
which is referred to as \textbf{plugging} in what follows. 
The associativity axiom for plugging and the neutrality effect of identity morphisms 
enable us to visualize the result of repeated pluggings 
as a rooted tree (Figure~\ref{rtree}). 

\begin{figure}
\input{rtree.tpc}
\caption{}
\label{rtree}  
\end{figure}

As in the case of ordinary category, defined are functors 
as well as natural transformations and natural equivalences between them. 
We say that 
two multicategories $\cM$ and $\cN$ are \textbf{equivalent} if 
we can find functors $F: \cM \to \cN$ and $G: \cN \to \cM$ 
so that their compositions $F\circ G$ and $G\circ F$ are 
naturally equivalent to identitiy functors. 
 

\begin{Example}
The multicategory $\cM\cS et$ of sets (and maps) and 
the multicategory $\cM\cV ec$ of vector spaces (and multilinear maps). 
\end{Example}

Given a (strict) monoidal category $\cC$, we define a multicategory $\cM$ so that 
$\cC$ and $\cM$ have the same class of objects and 
$\Hom(X_1\times \dots \times X_d,X) = \Hom(X_1\otimes \dots \otimes X_d,X)$. 

\begin{Proposition}
Let $\cC'$ be another monoidal category with $\cM'$ the associated multicategory. 
Then a multicategory-functor $\cM \to \cM'$ is in a one-to-one correspondence with 
a weakly monoidal functor $\cC \to \cC'$. Here by a weakly monoidal functor we shall mean 
a functor $F: \cC \to \cC'$ 
with a natural family of morphisms $m_{X,Y}: F(X)\otimes F(Y) \to F(X\otimes Y)$ 
satisfying the hexagonal identities for associativity. 
\end{Proposition}

\begin{proof}
Given a weakly monoidal functor $F: \cC \to \cC'$, we extend it to 
a multicategory-functor $\widetilde F: \cM \to \cM'$ by the composition 
\[
\widetilde F(T) = \left(
  \begin{CD}
F(X_1)\otimes \dots \otimes F(X_l) @>{m}>> F(X_1\otimes \dots \otimes X_l) 
@>{F(T)}>> F(X)
  \end{CD}
\right) 
\]
with $T \in \Hom_\cM(X_1,\dots, X_l; X) = \Hom_\cC(X_1\otimes \dots \otimes X_l,X)$. 
Then $\widetilde F$ is multiplicative: 
Let $T_j: X_{j,1}\otimes \dots \otimes X_{j,d_j} \to X_j$ ($j=1, \dots, l$) and consider 
the composition $T\circ(T_1, \dots, T_l)$. By definition, 
$\widetilde F(T)\circ({\widetilde F}(T_1),\dots,{\widetilde F}(T_l))$ is given by 
\[
\begin{CD}
F(X_{1,1})\otimes \dots \otimes F(X_{1,d_1}) \otimes \dots \otimes  
F(X_{l,1})\otimes \dots \otimes F(X_{l,d_l})\\ 
@V{m\otimes \dots \otimes m}VV\\ 
F(X_{1,1}\otimes \dots \otimes X_{1,d_1}) \otimes \dots \otimes 
F(X_{l,1}\otimes \dots \otimes X_{l,d_l})\\ 
@V{F(T_1)\otimes \dots \otimes  F(T_l)}VV\\
F(X_1)\otimes \dots \otimes F(X_l)\\
@V{m}VV\\
F(X_1\otimes \dots \otimes X_l)\\ 
@V{F(T)}VV\\
F(X)
\end{CD} 
\]
By the commutativity of 
\[
\begin{CD}
F(X_{1,1}\dots X_{1,d_1})\dots F(X_{l,1}\dots X_{l,d_l}) 
@>{m}>> F(X_{1,1}\dots X_{1,d_1}\dots X_{l,1}\dots X_{l,d_l})\\
@V{F(T_1)\otimes \dots \otimes F(T_l)}VV @VV{F(T_1\otimes \dots \otimes T_l)}V\\
F(X_1)\dots F(X_l) @>>{m}> F(X_1\dots X_l)
\end{CD} 
\]
and the identity $m\circ (m\otimes \dots \otimes m) = m$, 
$\widetilde F(T)\circ ({\widetilde F}(T_1),\dots, {\widetilde F}(T_l))$ is 
identical with the composition
\[
\begin{CD}
F(X_{1,1})\otimes \dots \otimes F(X_{1,d_1}) \otimes \dots \otimes  
F(X_{l,1})\otimes \dots \otimes F(X_{l,d_l})\\ 
@VVV\\
F(X_{1,1}\otimes \dots \otimes X_{1,d_1} \otimes \dots \otimes 
X_{l,1}\otimes \dots \otimes X_{l,d_l})\\ 
@VVV\\
F(X_1\otimes \dots \otimes X_l)\\
@VVV\\
F(X)
\end{CD}, 
\]
which is equall to $\widetilde F(T\circ(T_1,\dots, T_l))$. 

Conversely, starting with a multicategory-functor $\widetilde F: \cM \to \cM'$, 
let $F: \cC \to \cC'$ be the restriction of $\widetilde F$ and set 
\[
m_{X,Y} = \widetilde F(1_{X\otimes Y}): F(X)\otimes F(Y) \to F(X\otimes Y). 
\]
Here $1_{X\otimes Y}$ in the argument of $\widetilde F$ is regarded as a morphism 
in $\Hom_\cM(X,Y;X\otimes Y) = \End_\cC(X\otimes Y)$. 
The commutativity of 
\[
\begin{CD}
F(X)\otimes F(Y) @>{F(f)\otimes F(g)}>> F(X')\otimes F(Y')\\
@V{m_{X,Y}}VV @VV{m_{X',Y'}}V\\
F(X\otimes Y) @>>{F(f\otimes g)}> F(X'\otimes Y')
\end{CD}
\]
follows from 
\[
F\left( \input{dipod.tpc} \right) = \input{dipod2.tpc} 
\]
and the associativity of $m$ (the hexagonal identities), i.e., 
the commutativity of 
\[
\begin{CD}
F(X)\otimes F(Y)\otimes F(Z) @>{m_{X,Y}}>> F(X\otimes Y)\otimes F(Z)\\
@V{m_{Y,Z}}VV @VV{m_{X\otimes Y,Z}}V\\
F(X)\otimes F(Y\otimes Z) @>>{m_{X,Y\otimes Z}}> F(X\otimes Y\otimes Z)
\end{CD} 
\]
is obtained if we apply $\widetilde F$ to the identity
\[
\left( 
\input{tripod.tpc} 
\right) 
\quad
= 
\quad
\left( 
\input{tripod2.tpc} 
\right). 
\]
\end{proof}

\begin{Definition}
A (linear) \textbf{representation} of a multicategory $\cM$ is just 
a functor $F: \cM \to \cM\cV ec$. 
A representation is equivalently described in terms of a family $\{ V_X \}$ of vector spaces 
indexed by objects of $\cM$ together with a family of multilinear maps 
$\{ \pi_T: V_{X_1}\times \dots \times V_{X_d} \to V_X \}$ indexed by morphisms 
in $\cM$ (satisfying certain relations for multiplicativity). 

An intertwiner between two representations $\{\pi_T, V_X\}$ and 
$\{\pi'_T,  V'_X\}$ 
is defined to be a natural 
linear transformation, which is specified by a family of linear maps 
$\{ \varphi_X: V_X \to V'_X\}$ making the following diagram commutative 
for each morphism $T: X_1\times \dots \times X_d \to X$ in $\cM$: 
\[
\begin{CD}
V_{X_1}\times \dots \times V_{X_d} @>{\pi_T}>> V_X\\
@V{\varphi_{X_1}\times \dots \times \varphi_{X_d}}VV @VV{\varphi_X}V\\
V'_{X_1}\times \dots \times V'_{X_d} @>>{\pi'_T}> V'_X
\end{CD}
\]
If $\cM$ is a small multicategory (i.e., objects of $\cM$ form a set), 
representations of $\cM$ 
constitute a category $\cR ep(\cM)$ whose objects are representations and 
morphisms are intertwiners. 
\end{Definition}

Let $F: \cM \to \cN$ be a functor between small multicategories. 
By pulling back, we obtain a functor $F^*: \cR ep(\cN) \to \cR ep(\cM)$;  
given a representation $(\pi,V)$ of $\cN$, 
$F^*(\pi,V) = (F^*\pi, F^*V)$ is the representation of $\cM$ defined by 
$(F^*V)_X = V_{F(X)}$ and $(F^*\pi)_T = \pi_{F(T)}$. 

If $\phi: F \to G$ is a natural transformation 
$\{ \phi_X: F(X) \to G(X) \}$ with $G: \cM \to \cN$ 
another functor, it induces a natural transformation 
$\varphi:F^* \to G^*$: 
Let $(\pi,V)$ be a representation of $\cN$. Then 
$\varphi_{(\pi,V)}: F^*(\pi,V) \to G^*(\pi,V)$ 
is an intertwiner between representations of $\cM$ defined by 
the family $\pi(\phi) = \{ \pi_{\phi_X}: V_{F(X)} \to V_{G(X)} \}$ of linear maps. 

By the multiplicativity of $\pi$, 
the correspondence $\phi \to \varphi$ is multiplicative as well and 
the construction is summarized to be defining 
a functor $\cH om(\cM, \cN) \to \cH om(\cR ep(\cN), \cR ep(\cM))$. 

\begin{Proposition}
The family of functors 
\[
\cH om(\cM,\cN) \to \cH om(\cR ep(\cN), \cR ep(\cM))
\]
for various multicategories $\cM$ and $\cN$ defines a anti-multiplicative meta-functor of 
strict bicategories: $(F\circ G)^* = G^*\circ F^*$ for 
$F: \cM \to \cN$ and $G: \cL \to \cM$. 
\end{Proposition}

\begin{Corollary}
If small multicategories $\cM$ and $\cN$ are equivalent, then 
so are their representation categories $\cR ep(\cM)$ and $\cR ep(\cN)$. 
\end{Corollary}

\begin{proof}
If an equivalence between $\cM$ and $\cN$ is given by functors 
$F: \cM \to \cN$ and $G: \cN \to \cM$ with $F\circ G \cong \text{id}_{\cN}$ 
and $G\circ F \cong \text{id}_{\cM}$, then 
$G^*\circ F^* = (F\circ G)^* \cong \text{id}_{\cR ep(\cN)}$ 
and $F^*\circ G^* = (G\circ F)^* \cong \text{id}_{\cR ep(\cM)}$ 
show the equivalence between $\cR ep(\cM)$ and $\cR ep(\cN)$. 
\end{proof}

As observed in \cite{Gho}, the multicategory $\cM\cV ec$ admits 
a special object; the vector space of the ground field itself, 
which plays the role of unit when multiple objects are regarded 
as products. In the multicategory $\cM\cS et$, the special object in this sense 
is given by any one-point set. 
Multicategories of planar diagrams to be discussed shortly also admit 
such special objects; disks or boxes without pins. 
It is therefore natural to impose the condition that $V_S$ 
is equal to the ground field for a special object $S$. 

It is quite obvious to introduce other enhanced categories of similar flabor: 
co-multicategories and bi-multicategories with hom-sets indicated by 
\[
\Hom(X;X_1, \dots, X_d), 
\quad 
\Hom(X_1, \dots, X_m; Y_1, \dots, Y_n)
\]
respectively.

\section{Planar Diagrams}
We introduce several multicategories related with 
planar diagrams (namely, tangles without crossing points). 

\subsection{Disk Type}
Let $n$ be a non-negative integer. By a disk of type $n$ 
(or simply an $n$-disk), we shall mean a disk 
with $n$ pins attached on the peripheral and 
numbered consecutively from $1$ to $n$ anticlockwise. 

Our first example of multicategories has $n$-disks 
for various $n$ as objects with morphisms given by 
planar diagrams connecting pins inside the multiply punctured region of 
the target object (disk), Figure~\ref{region}. 
The multicategory obtained in this way is denoted by 
$\cD_\circ$ and called the multicategory of planar diagrams of disk type. 
The identity morphisms are given by diagrams consisting of spokes (Figure~\ref{spoke}). 

\begin{figure}[h]
\input{region.tpc}
\caption{}
\label{region}
\end{figure}
\begin{figure}[h]
\input{pot-spoke.tpc}
\caption{}
\label{spoke}
\end{figure}

\subsection{Box Type}
Let $m,n \in \N$ be non-negative integers. 
By a box of type $(m,n)$ or simply 
an $(m,n)$-box, we shall mean a rectangular box 
with $m$ pins and $n$ pins attached on 
the lower and upper edges respectively. 
Visually, the distinction of lower and upper edges can be 
indicated by putting an arrow from bottom to top. 

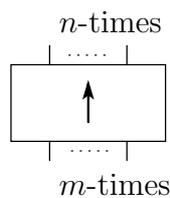
\begin{figure}[h]
\input{pot-box.tpc}
\end{figure}

The second example of multicategory has 
$(m,n)$-boxes for various $m, n$ as objects. 
For a pictorial description of morphisms, we distinguish 
boxes depending on whether it is used for outputs or inputs; 
outer or inner boxes, 
where pins are sticking out inward or outward respectively. 
For outer boxes, arrows are often omitted. 
When $m=n$, the box is said to be diagonal. 

By a \textbf{planar $(m,n)$-diagram} 
or simply an \textbf{$(m,n)$-diagram}, 
we shall mean 
a planar arrangement of inner boxes and curves (called strings) 
inside an outer $(m,n)$ box with 
each endpoint of strings connected to exactly one pin sticking out of 
inner or outer boxes so that no pins are left free. 
We shall not ditinguish two $(m,n)$-diagrams which are planar-isotopic.  

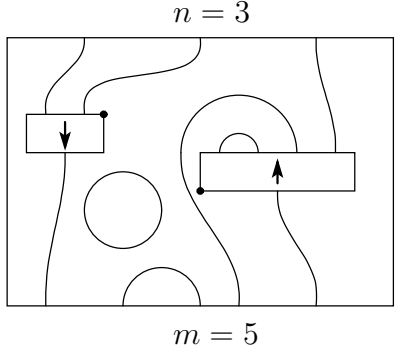
\begin{figure}[h]
\input{pot-tangle.tpc}  
\end{figure}

If inner boxes are distinguished by numbers $1, \dots, d$, we have a sequence of 
their types $((m_1,n_1), \dots, (m_d,n_d))$. When all relevant boxes 
are diagonal, the diagram is said to be diagonal. 

Multicategory morphisms are then given by planar diagrams
with the following operation of \textbf{plugging} (or nesting): 
Let $T$ be a planar $(m,n)$-diagram containing boxes of inner type $((m_j,n_j))_{1 \leq j \leq d}$ and 
$T_j$ be an $(m_j,n_j)$-diagram ($1 \leq j \leq d$). 
Then the plugging of 
$T_j$ into $T$ results in a new $(m,n)$-diagram, which is denoted by 
$T\circ(T_1\times \dots \times T_d)$. 
Note that the plugging produces diagonal planar diagrams out of diagonal ones. 

The plugging operation satisfies the associativity 
and we obtain a multicategory $\cD_\square$, 
which is referred to as 
the \textbf{multicategory of planar diagrams} of box type. 
Here identity morphisms are given by parallel vertical lines. 
Note that two objects (boxes) are isomorphic if and only if they have 
the same number $m+n$ of total pins. 

When objects are restricted to disks or boxes having even number of pins, 
we have submulticategories $\cD_\circ^{\text{even}}$ and 
$\cD_\square^{\text{even}}$. 

If boxes (objects) are further 
restricted to diagonal ones in $\cD_\square^{\text{even}}$, then we obtain 
another submulticategory $\cD_\triangle$  
as a subcategory of $\cD_\square$. 

\begin{Proposition}
Two multicategories $\cD_\circ$, $\cD_\square$ are equivalent. 
Three multicategories $\cD_\circ^{\text{even}}$, $\cD_\square^{\text{even}}$ and 
$\cD_\triangle$ are equivalent, whence they produce 
equivalent representation categories. 
\end{Proposition}

\begin{proof}
The obvious functors $\cD_\square \to \cD_\circ$, 
$\cD^{\text{even}}_\square \to \cD^{\text{even}}_\circ$ and 
$\cD_\triangle \to \cD^{\text{even}}_\square$ are fully faithful.
For example, to see the essential surjectivity of 
$\cD_\triangle \to \cD_\square^{\text{even}}$ on objects, 
given an object of $\cD_\square^{\text{even}}$labeled by $(m,n)$, 
let $S: (m,n) \to ((m+n)/2,(m+n)/2)$ and 
$T: ((m+n)/2,(m+n)/2) \to (m,n)$ be morphisms in $\cD_\square$ 
obtained by bending strings in the right vacant space. Then 
$S\circ T = 1_{(m+n)/2,(m+n)/2}$ and $T\circ S = 1_{m,n}$ show that 
$(m,n)$ and $((m+n)/2,(m+n)/2)$ are isomorphic as objects. 
\end{proof}

Here are three special plugging operations of special interest 
in $\cD_\square$; 
composition, juxtaposition and transposition. 

\textbf{Composition} (or product) produces 
an $(l,n)$-diagram $S T$ from 
an $(m,n)$-diagram $S$ and an $(l,m)$-diagram $T$: 
($(l,m,n) = (4,2,3)$ in the figure)

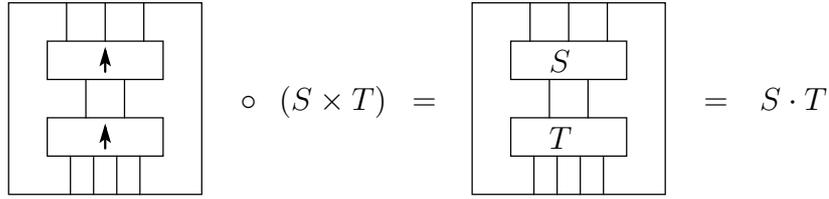
\begin{figure}[h]
\input{pot-composition.tpc}  
\end{figure}

The composition satisfies the associativity law and admits 
the identity diagrams for multiplication.

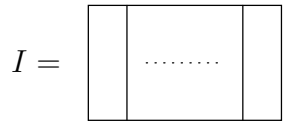
\begin{figure}[h]
\input{pot-unit.tpc}  
\end{figure}

In this way, we have found another categorical structure 
for planar diagrams of box type; 
the category $\cM$ has 
natural numbers $0, 1, 2, \dots$ as objects 
with hom sets $\cM(m,n)$ consisting of 
$(m,n)$-diagrams. 


\textbf{Juxtaposition} (or tensor product) produces an $(k+m,l+n)$-diagram 
$S\otimes T$ from an $(l,k)$-diagram $S$ and an $(m,n)$-diagram $T$. 

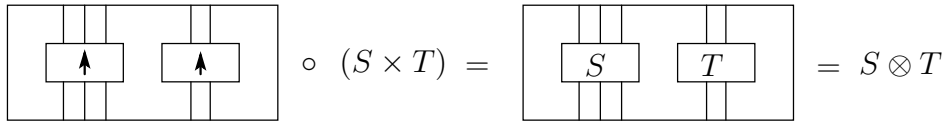
\begin{figure}[h]
\hspace{-9mm}
\input{pot-juxta.tpc}  
\end{figure}

With this operation, $\cM$ becomes a strict monoidal category 
($m\otimes n = m+n$). The unit object is $0$ with the identity morphism 
in $\cM(0,0)$ given by the empty diagram (neither inner boxes nor strings). 

Warning: monoidal categories connote multicategory structure as observed before, 
which is, however, 
different from $\cD_\square$: 
They have different classes of objects. 

\textbf{Transposition} is an involutive operation on planar diagrams of box type, 
which produces an $(n,m)$-diagram ${}^tT$ out of an $(m,n)$-diagram $T$. 

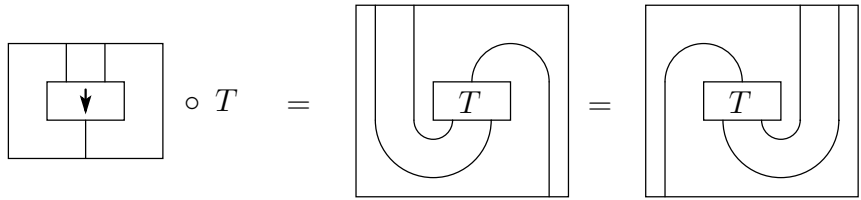
\begin{figure}[h]
\input{pot-trans.tpc}  
\end{figure}

Notice the last equality holds by  planar isotopy. 
Here are some obvious identities: 
\[
{}^t(ST) = {}^tT\,{}^tS, 
\qquad 
{}^t(S\otimes T) = {}^tT\otimes {}^tS. 
\]

With this operation, our monoidal category $\cM$ 
is furnished with a \textbf{pivotal} structure. 

From the definition, a representation of 
the multicategory $\cD_\square$ means 
a family of vector spaces $\{ P_{m,n} \}_{m, n \geq 0}$ together with 
an assignment of a linear map 
\[
\pi_T: P_{m_1,n_1} \otimes \dots \otimes P_{m_d,n_d} 
\to P_{m,n}
\]
to each morphism $T$ in $\cD_\square$, which satisfies 
\[
\pi_T(\pi_{T_1}(x_1)\otimes \dots \otimes \pi_{T_d}(x_d)) 
= \pi_{T\circ(T_1\times \dots \times T_d)}(x_1\otimes \dots \otimes x_d). 
\]

According to V.~Jones, this kind of algebraic structure is referred to 
as a \textbf{planar algebra}. 
In what follows, we use  the word `tensor category' to stand for a linear monoidal category. 

\begin{Proposition}
A representation $P = \{ P_{m,n} \}$ of $\cD_\square$ gives rise to 
a strict pivotal tensor category $\cP$ generated by a single 
self-dual object $X$: 
$\Hom(X^{\otimes m},X^{\otimes n}) = P_{m,n}$, composition of morphisms is given by 
$ab = \pi_C(a\otimes b)$, tensor product of morphisms is 
$a\otimes b = \pi_J(a\otimes b)$ and pivotal structure is given by transposition operation. 
(The identity morphisms are $\pi_I$.)
The construction is functorial and an intertwiner $\{ f_{m,n}: P_{m,n} \to P'_{m,n} \}$ 
between representations 
induces a monoidal functor $F: \cP \to \cP'$ preserving pivotality. 

Conversely, given a pivotal tensor category $\cP$ generated by  
a self-dual object $X$, 
we can produce a representation so that 
$P_{m,n} = \Hom(X^{\otimes m}, X^{\otimes n})$. 
\end{Proposition}

\begin{proof}
Since the monoidal structure is defined in terms of special forms of plugging, 
an intertwiner induces a monoidal functor. 

Conversely, suppose that we are 
given a pivotal tensor category with a generating object $X$. 
Let $\epsilon: X\otimes X \to I$ and $\delta: I \to X\otimes X$ give 
a rigidity pair satisfying $\epsilon = {}^t\delta$. 
Given a planar diagram $T$, let $\pi_T$ be a linear map 
obtained by replacing vertical parts, upper and 
lower arcs of strings with the identity, $\delta$ and $\epsilon$ respectively. 
Then the rigidity identities ensure that $\pi_T$ is unchanged under planar isotopy 
on strings if relevant boxes are kept unrotated, while the pivotality witnesses 
the planar isotopy for rotation of boxes (see Figure~\ref{eddy}). 
The multiplicativity of $\pi$ 
for plugging is now obvious from the construction. 
\end{proof}

\begin{figure}[h]
  \centering
  \input eddy.tpc
  \caption{}
\label{eddy}
\end{figure}

\begin{Remark}~ 
\begin{enumerate}
\item 
The condition $\dim P_{0,0} = 1$ is equivalent to the simplicity of 
the unit object of the associated tensor category. 
\item 
If one starts with a representation $P$ and make $\cP$, 
then the pivotal category $\cP$ produces $P$ itself. 
If one starts with a pivotal tensor category $\cP$ with $P$ 
the associated representation and let $\cQ$ be 
the pivotal category $\cQ$ constucted from $P$, then 
the obvious monoidal functor $\cQ \to \cP$ so that 
$n \mapsto X^{\otimes n}$ gives an equivalence of pivotal tensor 
categories (it may happen that $X^{\otimes m} = X^{\otimes n}$ 
in $\cP$ for $m \not= n$ though).  
\end{enumerate}
\end{Remark}

\begin{Example}
Let $K(m,n)$ be the set of Kauffman diagrams, i.e., planar $(m,n)$-diagrams with 
neither inner boxes nor loops. 
Recall that
$|K(m,n)|$ is the $(m+n)/2$-th Catalan number if $m+n$ is even and $|K(m,n)| = 0$ otherwise. 
Let $\C[K(m,n)]$ be a free vector space of basis set $K(m,n)$. Given a complex number 
$d$, we define a representation of $\cD$ by extending the obvious action of 
planar diagrams on $K(m,n)$ with each loop replaced by $d$. 
The resultant tensor category is the so-called 
\textbf{Temperley-Lieb category} and 
denoted by $\cK_d$ in what follows. 
(See \cite{Yam1} for more information.) 
\end{Example}

\begin{Example}
Let $\cT an(m,n)$ be the set of tangles and let 
$\C[\cT an(m,n)]$ be the free vector space generated by the set $\cT an(m,n)$. 
By extending the obvious action of planar diagrams on $\cT an$ to 
$\C[\cT an]$ linearly, we have a representation of $\cD$. 
Note that $\C[\cT an(0,0)]$ is infinite-dimensional. 
\end{Example}

\section{Decoration}
The previous construction allows us to have many variants if one  
assigns various attributes to strings and boxes. 
We here discuss two kinds of them, coloring and orientation, 
which can be applied independently (i.e., at the same time or seperately). 

To be explicit, let $C$ be a set 
and call an element of $C$ a color. 
By a \textbf{colored planar diagram}, we shall mean a planar diagram $T$ 
with a color assigned to each string. 
For colored planar diagrams, 
plugging is allowed only when 
color matches  at every connecting point. 

As before, colored planar diagrams constitute 
a multicategory $\cD_C$ 
whose objects are disks or boxes with pins decorated by colors. 
For colored planar diagrams of box type, 
a strict pivotal category 
$\cM_C$ is associated so that objects in $\cM_C$ are 
the words associated with the letter set $C$, 
which are considered to be upper or lower halves 
of decorations of boxes. 
In other words, objects in $\cD_C$ are labeled by pairs of 
objects in $\cM_C$. 

\begin{Example}
Let $K(v,w)$ ($v \in C^m, w \in C^n$ with $m, n \in \N$) be 
the set of colored Kauffman diagrams. Then, given a function $d: C \to \C$, 
we obtain a representation of $\cD_C$ by making free vector spaces $\C[K(v,w)]$ 
as in the Temperley-Lieb category. The resultant tensor category is denoted by 
$\cK_d$ and referred to as the \textbf{Bisch-Jones category}. 
\end{Example}

Given a colored planar diagram $T$, we can further decorate it 
by assigning orientations to each string in $T$. 
We call such a stuff  
a (planar) \textbf{oriented diagram}
(simply pod). The operation of plugging works here for 
colored pods and we obtain again a multicategory 
$\cO\cD_C$ of pods colored by $C$, 
where objects are disks or boxes with pins decorated by 
colors and orientations. 

Associated to colored pods of box type, 
we have a pivotal monoidal category $\cO\cM_C$ 
whose objects are 
words consisting of letters in $\{ c_+, c_-; c \in C \} 
= C\times \{+,-\}$ (for a pictorial display, we assign 
$+$ (resp.~$-$) to an upward (resp.~downward) arrow 
on boundaries of boxes). The product of objects is given by the concatenation of words 
with the monoidal structure for morphisms defined by the same way as before. 








Given a representation of $\cD_C$ or 
$\cO\cD_C$, 
we can construct a pivotal tensor category as before. 



\begin{Example}
For an object $x$ in $\cO\cD_C$, let 
$P_x$ be 
the free vector space (over a field) generated by the set 
\[
\bigsqcup_{d, x_1, \dots, x_d} 
\cO\cD_C(x_1\times\dots\times x_d,x) 
\]
of all colored pods having $x$ as a decoration of the outer box. 
If the plugging operation is linearly extended to these free vector spaces, 
we obtain a representation 
of $\cO\cD_C$, which is referred to as 
the \textbf{universal representation} 
because any representation of 
$\cO\cD_C$ splits through 
the universal one in a unique way. 
\end{Example}

Question: Is it possible to extract analytic entities out of the universal representation? 

\section{Half-Winding Decoration}
Related to the orientation, we here explain another kind of 
decoration on planar diagrams of box type according to \cite{FY}. 
To this end, we align directions of relevant boxes horizontally and 
every string 
(when attached to a box) perpendicular to the horizontal edges of a box. 
Let $p_0$ and $p_1$ be two end points of such a string and choose 
a smooth parameter 
$\varphi: [0,1] \to \R^2$ so that 
$\varphi(0) = p_0$ and $\varphi(1) = p_1$. 
By the assumption, 
$\displaystyle \frac{d\varphi}{dt}(0)$ and 
$\displaystyle \frac{d\varphi}{dt}(1)$ are vertical vectors. 
The \textbf{half-winding number} of the string from $p_0$ to $p_1$ is then  
an integer $w$ defined by 
\[
w = \frac{\theta(1) - \theta(0)}{\pi}, 
\]
where a continuous function $\theta(t)$ is introduced so that 
$\displaystyle \varphi(t) = \left| \frac{d\varphi}{dt}(t)\right| 
(\cos\theta(t),\sin\theta(t))$. 
Thus $w$ is even or odd according to 
$\dfrac{d\varphi}{dt}(0) \cdot \dfrac{d\varphi}{dt}(1) > 0$ or not. 

We now decorate boxes by assigning an integer to each pin. 
A diagram framed by such boxes 
is said to be \textbf{winding} if 
it contains no loops and each string with end points $p_0$ and $p_1$  satisfies 
\[
w = n_1 - n_0, 
\]
where $n_0$ and $n_1$ are integers attached to pins at $p_0$ and $p_1$ respectively. 

\begin{figure}[h]
  \centering
  \input{pot-halfwind.tpc}
  \caption{}
\label{half-wind}
\end{figure}

A diagram colored by a set $C$ is said to be winding if integers are assigned 
to relevant pins in such a way that the diagram is winding. 
It is immediate to see that winding diagrams in $\cD_C$ are closed under 
the operation of plugging (particularly, plugging does not produce 
loops out of winding diagrams) and we obtain 
a multicategory $\cW\cD_C$ of colored winding diagrams. 

By the following identification of left and right dual objects 
\[
(X,n) =
\begin{cases}
X^{*\dots *} &\text{if $n > 0$,}\\
X &\text{if $n=0$,}\\
^{*\dots *}X &\text{if $n < 0$,}
\end{cases}
\]
we have a one-to-one correspondence between 
representations of $\cW\cD_C$ and rigid tensor categories generated by objects 
labeled by the set $C$ as an obvious variant of the previous construction. 

Now the color set $C$ is chosen to consist of objects in a small linear category $\cL$ and  
we shall introduce a representation $\{ P_x\}$ 
of $\cW\cD_C$ ($x$ runs through objects of $\cW\cD_C$) as follows: 
$P_x = 0$ if the number of pins in $x$ is odd. 
To describe the case of even number pins, 
we consider 
a diagram of Temperley-Lieb type with its boundary decorated in a winding way and 
objects of $\cL$ assigned to the pins of the diagram, which 
is said to be admissible. To an admissible diagram $D$, we associate 
the vector space 
\[
\cL(D) = \bigotimes_j \cL_j,
\]
where $j$ indexes strings of the diagram and the vector space $\cL_j$ 
is determined by the following rule: 
If the $j$-th string connects a pin (colored by $a$) 
on a lower boundary and a pin (colored by $b$) on a upper boundary, 
then $\cL_j = \cL(a,b)$. 
When the $j$-th string connects pins on upper boundaries which 
are decorated by $(a,n)$ and $(b,n+1)$, 
we set 
\[
\cL_j = \begin{cases}
\cL(a,b) &\text{if $n$ is odd,}\\
\cL(b,a) &\text{if $n$ is even.}
\end{cases}
\]
When the $j$-th string connects pins on lower boundaries which 
are decorated by $(a,n)$ and $(b,n+1)$, 
we set 
\[
\cL_j = \begin{cases}
\cL(a,b) &\text{if $n$ is even,}\\
\cL(b,a) &\text{if $n$ is odd.}
\end{cases}
\]

Now set 
\[
P_{(a,k), (b,l)} = \bigoplus_D \cL(D).
\]
Here $D$ runs through winding diagrams having 
$(a,k) = \{ (a_j,k_j) \}$ and $(b,l) = \{ (b_j,l_j) \}$ 
as upper and lower decorations respectively. 

The rule of composition is the following: 

\begin{figure}[h]
  \centering
  \input{fy.tpc}
  \caption{}
\end{figure}

The figure~\ref{whirty} 
indicates that, though restrictive, the boundary decorations 
do not determine possible diagrams in a unique way. 

\begin{figure}[h]
  \centering
  \input{whirty.tpc}
  \caption{}
\label{whirty}
\end{figure}

\begin{Example}
If $\cL$ consists of one object $*$, then 
$\{ P_{(a,k),(b,l)}\}$ is the wreath product of the Temperley-Lieb 
category by the algebra $\cL(*,*)$ discussed in \cite{Jon}. 
\end{Example} 

The representation of $\cW\cD_C$ defined so far, in turn, gives rise to 
a rigid tensor category, which is denoted by $\cR[\cL]$. 
Note that $\cR[\cL]$ 
is not pivotal by the way of half-winding decoration. 

\begin{Proposition}[\cite{FY}, Theorem~3.8]
Let $\cR$ be a rigid tensor category and $F: \cL \to \cR$ be a linear functor. 
Then, $F$ is extended to a tensor-functor of $\cR[\cL]$ into $\cR$ in a unique way. 
\end{Proposition}

If the half-winding number indices are identified modulo $2$, we are reduced to 
the situation decorated by oriantation, i.e., a representation of $\cO\cM_C$. 
Let $\cP[\cL]$ be the associated pivotal tensor category.

\begin{Proposition}[\cite{FY}, Theorem~4.4]
Let $\cP$ be a pivotal tensor category and $F: \cL \to \cP$ be a linear functor. 
Then, $F$ is extended to a tensor-functor of $\cP[\cL]$ into $\cP$ in a unique way. 
\end{Proposition}

\begin{Remark}
If one replaces planar diagrams with tangles, 
analogous results are obtained on braided categories 
(\cite{FY}, Theorem~3.9 and Theorem~4.5). 
\end{Remark}

\section{Positivity}
We here work with planar diagrams of box type 
and use $v$, $w$ and so on to stand for an object 
in the associated monoidal category, whence any object of 
the multicategory is described by a pair $(v,w)$. 
Thus a representation space $P_{v,w}$ can be viewed 
as the hom-vector space of a tensor category. 
 
We now introduce two involutive operations on colored pods: 
Given a colored pod $T$, 
let $T'$ be the pod with the orientation of arrows reversed 
(colors being kept) and 
$T^*$ be the pod which is obtained as a reflection 
of $T'$ with respect to a horizontal line
(colors being kept while orientaions reflected).  

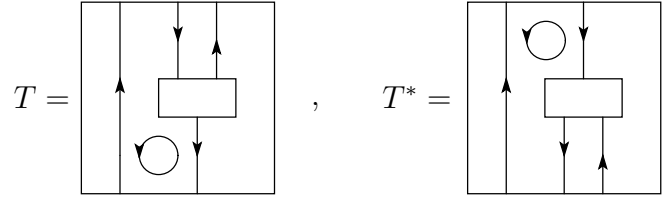
\begin{figure}[h]
\input{pot-star.tpc}  
\end{figure}

Here are again obvious identities: 
\[
({}^tT)^* = {}^t(T^*), 
\quad
(ST)^* = T^*S^*, 
\quad 
(S\otimes T)^* = S^*\otimes T^*.
\]

A representation $(\pi, \{ P_{v,w} \})$ of 
$\cO\cD_C$ is called a \textbf{*-representaion} if 
each $P_{v,w}$ is a complex vector space and 
we are given conjugate-linear involutions 
$*: P_{v,w} \to P_{w,v}$ satisfying 
\[
\pi_T(x_1,\dots, x_l)^* = \pi_{T^*}(x_1^*,\dots,x_l^*). 
\]

A *-representation is a \textbf{C*-representation} if 
\[
\begin{pmatrix}
P_{v_1,v_1} & \dots & P_{v_1,v_n}\\
\vdots & \ddots & \vdots\\
P_{v_n,v_1} & \dots & P_{v_n,v_n}
\end{pmatrix}
\]
is a C*-algebra for any finite sequence $\{ v_1,\dots, v_n \}$. 

\begin{Example}
The universal $\C$-representation of $\cO\cD_C$
is a C*-representation in a natural way. 
\end{Example}

\section{Alternating Diagrams} 
Consider now the category $\cO\cD$ of 
pods without coloring (or monochromatic coloring). 
Thus objects are finite sequences consisting of $+$ and $-$. 
We say that the decoration of a disk 
is \textbf{alternating} if 
even numbers of $\pm$ are arranged alternatingly; 
\[
(+,-,+, \cdots,+,-)\quad\text{or}\quad  
(-,+,-, \cdots, -,+). 
\]
By an alternating pod, we shall 
mean a pod where all boxes have even number of pins and 
are decorated by $\pm$ alternatingly and circularly. 
Thus orientations of strings attached to upper and lower boundaries of 
a box coincide at the left and right ends. 
Here are examples of alternating decorations on inner boxes: 

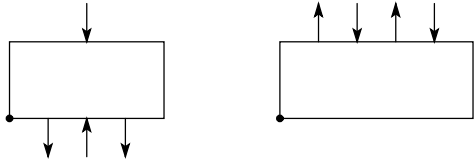
\begin{figure}[h]
\input{pot-altbox.tpc}  
\end{figure}

Alternating pods again constitute a multicategory, 
which is denoted by $\cA\cD$. 
According to the shape of objects, we have three equivalent 
categories $\cA\cD_\circ$, $\cA\cD_\square$ and 
$\cA\cD_\triangle$. 
So $\cA\cD$ is a loose notation to stand for one of 
these multicategories. 

If we further restrict objects to the ones whose decoration starts 
with $+$, then we obtain the submulticategory 
$\cA\cD^+$, which is the Jones' original form of planar diagrams:
A \textbf{planar algebra} is, by definition, a representation 
$\{ P_{n,n} \}_{n \geq 0}$ of $\cA\cD_\triangle^+$ 
satisfying $\dim P_{0,0} = 1$. 








We shall now deal with representations of 
$\cA\cD_+$ satisfying 
\[
\pi_T = d^l \pi_{T_0}, 
\]
where $d = d_-$ is a scalar, 
$l$ is the number of anticlockwise loops and 
$T_0$ is the pod obtained from $T$ by removing 
all the loops of anticlockwise orientation. 

\begin{figure}[h]
\input{pot-circle.tpc}  
\end{figure}

\begin{Lemma}
Under the assumption that $d \not= 0$, 
any representation of $\cA\cD_+$ 
is extended to a representation 
of $\cA\cD$ and the extension is unique. 
\end{Lemma}

\begin{proof}
Assume that we are given a representation $(\pi, P)$ of $\cA\cD$. 
According to the parity of label objects, the representation space 
$P$ is split into two families $\{ P_{m,n}^\pm \}$. 
Let $C$ be a pod in $\cA\cD$ indicated by Figure~\ref{cpot}. 
\begin{figure}
\input{cpot.tpc}
\caption{}
\label{cpot} 
\end{figure}
From the identity $\pi_C(1\otimes a) = d a$ for $a \in P^-_{m,n}$, 
we see that the map $P_{m,n}^- \ni a \mapsto 1\otimes a \in P_{m+1,n+1}^+$ 
is injective with its image specified by 
\[
1\otimes P_{m,n}^- = \{ a \in P_{m+1,n+1}^+; \pi_{1\otimes C}(a) = d a \}. 
\]
If we regard $P_{m,n}^- \subset P_{m+1,n+1}^+$ by this imbedding, 
$\pi_T$ for a morphism $T \in \cA\cD$ is identified with 
\[
\frac{1}{d^e} \pi_{T\circ(C_1^*\times \dots \times C_d^*)} 
\quad\text{or }\ 
\frac{1}{d^e} \pi_{(1\otimes T)\circ(C_1^*\times \dots \times C_d^*)} 
\]
depending on the parity of the output object of $T$. 
Here $C_j^* = 1$ or $C_j^* = C$ according to the parity of the $j$-th inner box and 
$e$ denotes the number of inner boxes of odd (= negative) parity in $T$. 
Note that these reinterpreted $T$'s are morphisms in $\cA\cD_+$. 
In this way, we have seen that $\pi$ is determined by the restriction to 
$\cA\cD_+$. 

Conversely, starting with a representation $(\pi^+, P^+)$ of $\cA\cD_+$, 
we set 
\[
P_{m,n}^- = \{ \pi^+_{1\otimes C}(a); a \in P_{m+1,n+1}^+\} \subset 
P_{m+1,n+1}^+
\]
and define a multilinear map $\pi_T$  by the above relation: 
\[
\pi_T = \frac{1}{d^e} \pi^+_{T\circ(C_1^*\times \dots \times C_d^*)} 
\quad 
\text{or }\ 
\frac{1}{d^e} \pi^+_{(1\otimes T)\circ(C_1^*\times \dots \times C_d^*)}. 
\]
From the definition, $\pi_T = \pi^+_T$ if $T$ is a morphism in $\cA\cD_+$. 

To see that $\pi$ is a representation of $\cA\cD$, we need to show that 
$\pi_T\circ(\pi_{T_1}\otimes \dots \otimes \pi_{T_d}) = 
\pi_{T\circ(T_1\times \dots \times T_d)}$. 

When the output object of $T$ has even parity, 
\[
\pi_T\circ(\pi_{T_1}\otimes \dots \otimes \pi_{T_d}) = 
\frac{1}{d^e} \pi^+_{T\circ(C_1^*\times \dots \times C_d^*)} 
\circ(\pi_{T_1}\otimes \dots \otimes \pi_{T_d})
\]
and we look into the plugging at the box such that $C_j^* = C$. 
Then the output parity of $T_j$ is odd and we have 
$\pi_{T_j} = d^{-e_j}\pi^+_{(1\otimes T_j)\circ(C^*\times \dots \times C^*)}$, 
which is used in the above plugging (Figure~\ref{plug}) 
to see that it results in 
\[
\frac{1}{d^e} \pi^+_{T\circ(C_1^*\times \dots \times C_d^*)} \circ 
(\pi_{T_1}\otimes \dots \otimes \pi_{T_d}) 
= \frac{1}{d^f} \pi_{T\circ(T_1\times \dots \times T_d)\circ (C^*\times \dots \times C^*)}, 
\]
where $f = \sum_j e_j$ denotes the number of inner boxes of odd parity 
inside $T_1, \dots, T_d$. 

A similar argument works for $T$ having the outer box of odd parity, 
proving the associativity of $\pi$ for plugging. 
\begin{figure}
\input{plug.tpc}
\caption{}
\label{plug}
\end{figure}
\end{proof}

\begin{Theorem}
Representations of $\cA\cD$ are in one-to-one 
correspondence with 
singly generated pivotal linear bicategories. 
\end{Theorem}

\begin{Corollary}
Planar algebras are in one-to-one correspondence 
with singly generated pivotal linear bicategories 
with simple unit objects and satisfying $\text{l-}\dim(X) \not= 0$. 
($\text{l-}\dim$ refers to the left dimension.)
\end{Corollary}

\begin{Corollary}
Planar C*-algebras are in one-to-one correspondence 
with singly generated rigid C*-bicategories 
with simple unit object. 
\end{Corollary}

\end{document}